\documentclass{amsart}
\usepackage[utf8]{inputenc}
\usepackage{amsfonts,amssymb,amscd,amsmath,enumerate,verbatim,calc}
\usepackage{xy}
\usepackage{mathtools}

\newcommand{\CM}{Cohen-Macaulay}

\newcommand{\m}{\mathfrak{m} }

\newcommand{\rt}{\rightarrow}

\newcommand{\sub}{\subseteq}

\newcommand{\depth}{\operatorname{depth}}

\theoremstyle{plain}

\newtheorem{theorem}{Theorem}[section]
\newtheorem{corollary}[theorem]{Corollary}
\newtheorem{lemma}[theorem]{Lemma}
\newtheorem{proposition}[theorem]{Proposition}

\theoremstyle{definition}
\newtheorem{definition}[theorem]{Definition}
\newtheorem{s}[theorem]{}
\newtheorem{remark}[theorem]{Remark}

\theoremstyle{remark}

\begin{document}
	\title[Associated graded modules]{On associated graded modules of maximal Cohen-Macaulay modules over hypersurface rings-II }
	\author{Ankit Mishra}
	\email{ankitmishra@math.iitb.ac.in}
	
	\author{ Tony~J.~Puthenpurakal}
	\email{tputhen@math.iitb.ac.in}
	
	\address{Department of Mathematics, IIT Bombay, Powai, Mumbai 400 076}

	\date{\today}
	
	
	\begin{abstract}
		If $(A,\mathfrak{m})$ is a  hypersurface ring of dimension $d$ with  $e(A)=3$. Let $M$ be an MCM $A$-module with $\mu(M)=4$ then we  prove that  $\depth{G(M)}\geq d-3$.
		\end{abstract}
	\maketitle
	\section{Introduction}
	
	Let $(A,\mathfrak{m})$ be Noetherian local ring of dimension $d$ and $M$ a finite \CM\ $A$-module of dimension $r$. Let $G(A)=\bigoplus_{n\geq0}\mathfrak{m}^n/\mathfrak{m}^{n+1}$ be associated graded ring of $A$ with respect to $\mathfrak{m}$ and $G(M)=\bigoplus_{n\geq0}{\mathfrak{m}^nM}/{\mathfrak{m}^{n+1}M}$ be associated graded module of $M$ with respect to $\mathfrak{m}$. Now $\mathcal{M}=\bigoplus_{n\geq1}\mathfrak{m}^n/\mathfrak{m}^{n+1}$ is irrelevant maximal ideal of $G(A)$ we set  $\depth{G(M)}$ = grade$(\mathcal{M},G(M))$. If $L$ is an $A$-module then minimal number of generators of $L$ is denoted by  $\mu(L)$ and its length is denoted by $\ell(L)$.
	
	The Hilbert-Samuel function of $M$ with respect to $\mathfrak{m}$ is
	$$H^1(M,n)=\ell({M}/{\mathfrak{m}^{n+1}M})\ \text{for all}\ n\geq0.$$
	There exists a polynomial $P_M(z)$ of degree $r$  such that
	  $$H^1(M,n)=P_M(n)\ \text{for}\ n\gg0.$$
	This polynomial can be written as $$P_M(X)=\sum_{i=0}^{r}(-1)^ie_i(M)\binom{X+r-i}{r-i}$$
	These coefficients $e_i(M)'$s are integers and known as {\it Hilbert coefficients} of $M$. Note that $e_0(M)$ is known as the {\it multiplicity} of $M$ and we denote it as $e(M)$.
	
	We know that Hilbert series of $M$ is formal power series $$H_M(z)=\sum_{n\geq0}\ell(\mathfrak{m}^nM/\mathfrak{m}^{n+1}M)z^n$$
	We can write $$ H_M(z)=\frac{h_M(z)}{(1-z)^r},\ \text{where}\ r=dimM $$
	
	Here, $h_M(z)=h_0(M)+h_1(M)z+\ldots+h_s(M)z^s\in \mathbb{Z}[z]$ and $h_M(1)\neq0$. This polynomial is know as {\it h-polynomial} of $M$.

	If we set $f^{(i)}$ to denote $i$th formal derivative of  a polynomial $f$ then it is easy to see that $e_i(M)=h_M^{(i)}(1)/i!$ for $i=0,\ldots,r$.  It is also  convenient to set $e_i(M)=h_M^{(i)}(1)/i!$ for all $i\geq0.$
	
	Now we know that if $(A,\mathfrak{m})$ is \CM\ with red$(A)\leq2$ (for definition see \ref{redno}) then $G(A)$ is \CM\ (see\cite[Theorem 2.1]{S}). 
	
	If $M$ is a \CM \ $A$-module with red$(M)\leq1$ then  $G(M)$ is \CM \ (see \cite[Theorem 16]{Pu0}), but if red$(M)=2$, then $G(M)$ need not be \CM\ (see \cite[Example 3.3]{PuMCM}).
	
	Here  we consider maximal \CM\ (MCM) modules over a \CM\ local ring $(A,\mathfrak{m})$. We know that if $A$ is a regular local ring then $M$ is free, say $M\cong A^s$. This implies $G(M)\cong G(A)^s$ is \CM.
	
	The next case is when $A$ is a hypersurface ring. {\it For convenience in the introduction we assume $A=Q/(f)$ where $(Q,\mathfrak{n})$ is a regular local ring with infinite residue field and $f\in \mathfrak{n}^2$.} 
	
	If $f\in \mathfrak{n}^2\setminus\mathfrak{n}^3$ then $A$ has minimal multiplicity. It follows that any MCM module $M$ over $A$ has minimal multiplicity. So $G(M)$ is \CM.
	
	We are interested in the case when  $f\in \mathfrak{n}^3\setminus\mathfrak{n}^4$. Note in this case red$(A)=2$. So if $M$ is any MCM $A$-module then red$(M)\leq2$. In this case $G(M)$ need not \CM\ (see  \cite[Example 3.3]{PuMCM}).

	Notice if $M$ is an MCM module over $A$ then projdim$_Q(M)=1$. So, $M$ has a minimal presentation over $Q$ 
	$$0\rt Q^{\mu(M)}\xrightarrow{\phi}Q^{\mu(M)}\rt M\rt 0.$$
	We investigate $G(M)$ in terms of invariants of a minimal presentation of $M$ over $Q$.

	For $\mu(M)=2,3$ we have proved that $\depth{G(M)}\geq d-\mu(M)+1$ (see \cite{Mishra}). Here we consider the next case that is the case when $\mu(M)=4$ and prove:

\begin{theorem}\label{3}
	Let $(A,\mathfrak{m})$ be a   hypersurface ring of dimension $d$ with $e(A)=3$. Let  $M$ be an MCM $A$-module. Now if   $\mu(M)=4$, then  $\depth{G(M)}\geq d-3$.
\end{theorem}

Here is an overview of the contents of this paper. In  section 2, we give some preliminaries which we have used in the paper.  In section 3, we prove Theorem \ref{3}. In the last section some examples, illustrating our results, are given.

\section{Priliminaries}
Let $(A,\mathfrak{m})$ be a Noetherian local ring of dimension $d$, and $M$ an $A$-module of dimension $r$.
\s An element $x\in \mathfrak{m}$ is said to be a superficial element of $M$ if there exists an integer $n_0>0$ such that $$(\mathfrak{m}^nM:_Mx)\cap \mathfrak{m}^{n_0}M=\mathfrak{m}^{n-1}M\ \text{for all}\ n>n_0$$
We know that if residue field $k=A/\mathfrak{m}$ is infinite then superficial elements always exist (see \cite[Pg 7]{Sbook}). A sequence of elements $x_1,\ldots,x_m$ is said to be superficial sequence if $x_1$ is $M$-superficial and $x_i$ is $M/(x_1,\ldots,x_{i-1})M$-superficial for $i=2,\ldots,m.$

\begin{remark}
	
	\begin{enumerate}
		\item 	If $x$ is $M-$superficial and regular then we have $(\mathfrak{m}^nM:_Mx)=\mathfrak{m}^{n-1}M$ for all $n\gg 0.$
		
		\item If $\depth{M}>0$ then it is easy to show that  every $M$-superficial element is also $M-$ regular.
	\end{enumerate}

\end{remark}

\s \label{Base change} Let $f:(A,\mathfrak{m})\rt (B,\mathfrak{n})$ be a flat local  ring homomorphism with $\mathfrak{m}B=\mathfrak{n}$. If $M$ is an $A$-module set $M'=M\otimes_A B$, then  following facts are well known 
\begin{enumerate}
	\item $H(M,n)= H(M',n)$ for all $n\geq0$.
	\item depth$_{G(A)}G(M)=$ depth$_{G(A')}G(M')$.
	\item projdim$_A(M)$ = projdim$_{A'}(M')$.
	
\end{enumerate}
We will use this result in the following two cases:
\begin{enumerate}[(a)]
	\item We can assume $A$ is complete by taking $B=\hat{A}$.
	\item We can assume the residue field of $A$ is infinite, because if the residue field $(k=A/\mathfrak{m})$ is finite we can take $B=A[X]_S$  where $S=A[x]\setminus \mathfrak{m}A[X]$. Clearly, the residue field of $B=k(X)$ is infinite.
\end{enumerate}

\s \label{compl_and_inf} Let $(A,\mathfrak{m})$ be a hypersurface ring. We can assume $A$ is complete (see \ref{Base change}(a)). So $A\cong Q/(f)$, where $(Q,\mathfrak{n})$ is a regular local ring and $f\in \mathfrak{n}^2$. If residue field of $A$ is finite take $B=A[X]_S$  where $S=A[x]\setminus \mathfrak{m}A[X]$. Note that $B$ is a quotient of a regular local ring by a principal ideal and residue field of $B$ is infinite.

 All the properties we deal in this article are invariant when we go from $A$ to $A'$. So we can assume that residue field of $A$ is infinite.

\s If $a$ is a non-zero element of $M$ and if $i$ is the largest integer such that $a\in \mathfrak{m}^iM$, then we denote image of $a$ in $\mathfrak{m}^i \  M/\mathfrak{m}^{i+1} \ M$ by $a^*$. If $N$ is a submodule of $M$, then $N^*$ denotes the graded submodule of $G(M)$ generated by all $b^*$ with $b\in N$.

\begin{definition}
	Let $(A,\mathfrak{m})$ be a Noetherian local ring and $M\ne 0$ be a finite $A$-module then $M$ is said to be a \CM\ $A$-module if  $\depth{M}=$ dim $M$, and a maximal \CM\ (MCM) module if  $\depth{M}=$ dim $A$.
\end{definition}

\s \label{mod-sup} If $x\in \mathfrak{m}\setminus\mathfrak{m}^2$ an $M-$superficial and regular element. Set $N=M/xM$ and $K=\mathfrak{m}/(x)$ then we have {\bf Singh's equality}\index{Singh's equality} ( for $M=A$ see \cite[Theorem 1]{singh}, and for the module case see \cite[Theorem 9]{Pu0})
$$H(M,n)=\ell(N/K^{n+1}N)-\ell\left(\frac{\mathfrak{m}^{n+1}M:x}{\mathfrak{m}^nM}\right)\ \text{for all}\ n\geq0.$$

Set $b_n(x,M)=\ell(\mathfrak{m}^{n+1}M:x/\mathfrak{m}^nM)$ and $b_{x,M}(z)=\sum_{n\geq0}b_n(x,M)z^n$. Notice that $b_0(x,M)=0$. Now we have 
$$h_M(z)=h_N(z)-(1-z)^rb_{x,M}(z)$$

\s \label{Property} (See \cite[Corollary 10]{Pu0}) Let $x\in \mathfrak{m}$ be an $M-$superficial and regular element. Set $B=A/(x)$, $N=M/xM$ and $K=\mathfrak{m}/(x)$ then we have
\begin{enumerate}
	\item dim$M-1$ = dim$N$ and $h_0(N)=h_0(M)$.
	\item $b_{x,M}$ is a polynomial.
	\item $h_1(M)=h_1(N)$ if and only if $\mathfrak{m}^2M\cap xM=x\mathfrak{m}M.$
	\item $e_i(M)=e_i(N)$ for $i=0,\ldots,r-1.$
	\item $e_r(M)=e_r(N)-(-1)^r\sum_{n\geq0}b_n(x,M).$
	\item $x^*$ is $G(M)$-regular if and only if $b_n(x,M)=0$ for all $n\geq0.$
	\item $e_r(M)=e_r(N)$ if and only if $x^*$ is $G(M)$-regular.
	\item  $\depth{G(M)}\geq1$ if and only if $h_{M}(z)=h_N(z)$.
\end{enumerate}
\s \label{Sally-des} {\bf Sally-descent }(see \cite[Theorem 8]{Pu0}): Let $(A,\mathfrak{m})$ be a \CM\ local ring of dimension $d$ and $M$ be \CM\ module of dimension $r$. Let $x_1,\ldots,x_c$ be a $M$-superficial sequence with $c\leq r-1$. Set $M_c=M/(x_1,\ldots,x_c)M$ then   
 $\depth{G(M)}\geq c+1$ if and only if  $\depth{G(M_c)}\geq 1$. 
\s \label{redno} The reduction number of $M$ (denoted as red$(M)$) can be defined as the least integer $\ell$ such that there is an ideal $J$ generated by a maximal superficial sequence with $\mathfrak{m}^{\ell+1}M=J\mathfrak{m}^{\ell}M$.

\begin{definition}\label{Ulrich}
	Let $(A,\mathfrak{m})$ be a Noetherian local ring and $M$ be a maximal \CM \ module  then $M$ is said to be a Ulrich module if $e(M)=\mu(M)$.

\end{definition}
\begin{remark}
	When $M$ is an MCM module and $\mathfrak{m}$ has a minimal reduction $J$ generated by a system of parameters, then $M$ is Ulrich module if and only if $\mathfrak{m}M=JM$. 
\end{remark}

\s  (See \cite[section 6]{heinzer}) For any $n\geq1$ we can define Ratliff-Rush submodule of $M$ associated with $\mathfrak{m}^n$ as 
$$\widetilde{\mathfrak{m}^nM}=\bigcup_{i\geq0}(\mathfrak{m}^{n+i}M:_M\mathfrak{m}^i)$$ The filtration $\{\widetilde{\mathfrak{m}^nM}\}_{n\geq1}$ is known as the Ratliff-Rush filtration \index{Ratliff-Rush filtration} of $M$ with respect to $\mathfrak{m}$.

For the proof of the following properties in the ring case see \cite{Ratliff}. This proof can be easily extended for the modules. Also see \cite[2.2]{Naghipour}.

\s If $\depth{M}>0$ and  $x\in \mathfrak{m}$ is a $M-$superficial element then we have
\begin{enumerate}
	\item $\widetilde{\mathfrak{m}^nM}=\mathfrak{m}^nM$ for all $n\gg0.$
	\item  $(\widetilde{\mathfrak{m}^{n+1}M}:x)=\widetilde{\mathfrak{m}^nM}$ for all $n\geq1.$
\end{enumerate}
\s\label{htilde} Let $\widetilde{G(M)}=\bigoplus_{n\geq0}\widetilde{\mathfrak{m}^nM}/\widetilde{\mathfrak{m}^{n+1}M}$ be the associated graded module  of $M$ with respect to Ratliff-Rush filtration. Then its Hilbert series
$$\sum_{n\geq0}\ell(\widetilde{\mathfrak{m}^nM}/\widetilde{\mathfrak{m}^{n+1}M})z^n=\frac{\widetilde{h_M}(z)}{(1-z)^r}$$
Where $\widetilde{h_M}(z)\in \mathbb{Z}[z]$. Set $r_M(z)=\sum_{n\geq0}\ell(\widetilde{\mathfrak{m}^{n+1}M}/\mathfrak{m}^{n+1}M)z^n$; clearly, $r_M(z)$ is a polynomial with non-negative integer coefficients (because $\depth{M}>0$). Now we have (see \cite[1.5]{Pu2}) $$h_M(z)=\widetilde{h_M}(z)+(1-z)^{r+1}r_M(z);\ \text{where }\ r=\text{dim}M$$
We know that $\depth{G(M)}>0$ if and only if $r_M(z)=0$.

\s (see \cite[2.1]{Pu2}) Let $x$ be an $M-$superficial element and depth$M\geq2$. Set $N=M/xM$, then we have a natural map $\rho^x:M\rt N$ and we say that Ratliff-Rush filtration on $M$ behaves well mod superficial element $x$ if $\rho^x(\widetilde{\mathfrak{m}^nM})=\widetilde{\mathfrak{m}^nN}$ for all $n\geq 1$. Now $\rho^x$ induces the maps $$\rho_n^x:\frac{\widetilde{\mathfrak{m}^nM}}{\mathfrak{m}^nM}\rt \frac{\widetilde{\mathfrak{m}^nN}}{\mathfrak{m}^nN}$$
It is easy to  show that Ratliff-Rush filtration behaves well mod $x$ if and only if $\rho_n^x$ is surjective for all $n\geq 1$.

\begin{definition}
	Let $(A,\mathfrak{m})$ be a Noetherian local ring and $M$ be a finite $A$-module with dim$M=d$. Then we say $G(M)$ is a generalized \CM \ $G(A)$-module if 
	$$\ell(H^i_\mathcal{M}(G(M)))<\infty\ \text{ for } \ i=0,\ldots,d-1$$
	where, $H^i_\mathcal{M}(G(M))$ is the $i$-th local cohomology module of $G(M)$ with respect to the maximal homogeneous ideal $\mathcal{M}$ of $G(A)$. 
\end{definition}
\begin{remark}
	$G(A)$ is a finitely generated $k(=A/\mathfrak{m})$-algebra. A $G(A)$-module $E$ is generalized \CM \ if and only if $E_P$ is \CM\ for all prime ideals $P\ne \mathcal{M}$.
\end{remark}

\begin{proposition}\label{ASSG}
	Let $(A,\mathfrak{m})$ be a  \CM\ local ring of dimension $d\geq 1$ and $M$ a finite $A$-module with dim$M=d$. Now if  $\widetilde{G(M)}$ is a \CM \ $G(A)$-module, then 
	\begin{enumerate}
		\item $G(M)$ is a generalized \CM\ module.
		\item dim $G(A)/P = d$ for all minimal primes $P$ of $G(M)$.
	\end{enumerate}
\begin{proof}
	Similar to the proof of \cite[Proposition 2.20]{Mishra1}.
	\end{proof}	
\end{proposition}

\begin{definition}\label{hypersurface}
	Let $(A,\mathfrak{m})$ be a Noetherian local ring, then $A$ is said to be a hypersurface ring if its completion can be written as a quotient of a regular local ring by a principal ideal.
\end{definition}
\s Let $(Q,\mathfrak{n})$ be a regular local ring, $f\in \mathfrak{n}^e\setminus\mathfrak{n}^{e+1}$ and $A=Q/(f)$. If $M $ is an MCM $A-$module then projdim$_Q(M)=1$ and $M$ has a minimal presentation: $$0\rt Q^{\mu(M)}\rt Q^{\mu(M)}\rt M \rt 0$$

\s \label{i(M)}Let $(Q,\mathfrak{n})$ be a regular local ring and $\phi : Q^t\rt Q^t$  a linear map,  set
$$i_\phi=\text{max}\{i |\  \text{all entries of}\ \phi \ \text{are in }\ \mathfrak{n}^i\}$$
If $M $ has minimal presentations: $0\rt Q^t\xrightarrow{\phi}Q^t\rt M\rt0$ and
$0\rt Q^t\xrightarrow{\phi'}Q^t\rt M\rt0$ then it is well known that $i_\phi=i_{\phi'}$ and det$(\phi)=u$det$(\phi')$ where $u$ is a unit. We set $i(M)=i_\phi$ and det$M=$ det$(\phi)$. For any non-zero element $a$ of $Q$ we set $v_Q(a)=max\{i|a\in \mathfrak{n}^i\}$. We are choosing this set-up from \cite{PuMCM}.

\begin{definition} \label{phi}
	(See \cite[Definition 4.4]{PuMCM}) Let $(Q,\mathfrak{n})$ be a regular local ring, $A=Q/(f)$ where $f\in \mathfrak{n}^e\setminus\mathfrak{n}^{e+1}, e\geq2$ and $M$  an MCM $A-$module with minimal presentation: $$0\rt Q^t\xrightarrow{\phi}Q^t\rt M\rt0$$
	Then an element $x$ of $\mathfrak{n}$ is said to be $\phi-$ superficial \index{$\phi-$ superficial}if we have
	\begin{enumerate}
		\item $x$ is $Q\oplus A\oplus M$ superficial.
		\item If $\phi=(\phi_{ij})$ then $v_Q(\phi_{ij})=v_{Q/xQ}(\overline{\phi_{ij}})$.
		\item $v_Q(det(\phi))=v_{Q/xQ}(det(\overline{\phi}))$
	\end{enumerate}
\end{definition}
\begin{remark}

	If $x$ is $Q\oplus A\oplus M\oplus (\oplus_{ij}Q/(\phi_{ij}))\oplus Q/(det(\phi))-$superficial then it is $\phi-$superficial. So if the residue field of $Q$ is infinite then $\phi-$superficial elements always exist.
\end{remark}
\begin{definition}
	(See \cite[Definition 4.5]{PuMCM}) Let $(Q,\mathfrak{n})$ be a regular local ring, $A=Q/(f)$ where $f\in \mathfrak{n}^e\setminus\mathfrak{n}^{e+1}, e\geq2$ and $M$  an MCM $A-$module with minimal presentation: $$0\rt Q^t\xrightarrow{\phi}Q^t\rt M\rt0.$$
	We say that $x_1,\ldots,x_c$ is a $\phi$-superficial sequence if $\overline{x_n}$ is $(\phi \otimes_Q Q/(x_1,\ldots,x_{n-1}))$-superficial for $n=1,\ldots,c$.
\end{definition}

\s \label{d=1} With above set-up we have   
\begin{enumerate}
	\item (\cite[Lemma 4.7]{PuMCM})  If dim$M=1$  then
	$$h_M(z)=\mu(M)(1+z+\ldots+z^{i(M)-1})+\sum_{i\geq i(M)}h_i(M)z^i$$
	$$ \text{with}\ h_i(M)\geq0 \ \forall \ i.$$
	
	\item (\cite[Theorem 2]{PuMCM}) $e(M)\geq \mu(M)i(M)$ and if $e(M)=\mu(M)i(M)$ then 
	
	$G(M)$ is Cohen-Macaulay and $h_M(z)=\mu(M)(1+z+\ldots+z^{i(M)-1})$.
	
\end{enumerate}
\s (\cite[Proposition 13]{Pu0}) \label{rho_formula} Let $(A,\mathfrak{m})$ be a \CM\ local ring and $M$ be a \CM\ module of dimension one. Let $x$ be an $M$-superficial element. Set $\rho_n(M)=\ell(\mathfrak{m}^{n+1}M/x\mathfrak{m}^nM)$ for all $n\geq 0$. If $\deg h_M(z)=s$ then $\rho_n(M)=0$ for all $n\geq s$ and $$h_M(z)=\mu(M)+\sum_{i=1}^{s}(\rho_{i-1}(M)-\rho_{i}(M))z^i.$$

\s \label{exact seq}Let $(A,\mathfrak{m})$ be a \CM\ local ring  and $M$ a \CM\ $A$-module of dimension 2. Let $x,y$ be a maximal $M$-superficial sequence. 

Set $J=(x,y)$ and $\overline{M}=M/xM$ then we have exact sequence (for $M=A$ see \cite[Lemma 2.2]{rv})
\begin{align*}
0 \rt \frac{\mathfrak{m}^nM:J}{\mathfrak{m}^{n-1}M}\xrightarrow{f_1} \frac{\mathfrak{m}^nM:x}{\mathfrak{m}^{n-1}M}  \xrightarrow{f_2} \frac{\mathfrak{m}^{n+1}M:x}{\mathfrak{m}^{n}M}
\xrightarrow{f_3} \frac{\mathfrak{m}^{n+1}M}{J\mathfrak{m}^nM}
\xrightarrow{f_4} \frac{\mathfrak{m}^{n+1}\overline{M}}{y\mathfrak{m}^n\overline{M}}\rt 0
\end{align*}
Here, $f_1$ is inclusion map, $f_2(a+\mathfrak{m}^{n-1}M)=ay+\mathfrak{m}^nM, f_3(b+\mathfrak{m}^nM)=bx+J\mathfrak{m}^nM$ and $f_4$ is reduction modulo $x$.

\s \label{exact d}Let $(A,\mathfrak{m})$ be a \CM\ local ring of dimension $d\geq 1$ and $M$ a maximal \CM\ $A$-module. Let $\underline{x}=x_1,\ldots,x_d$ be a maximal $M$-superficial sequence. Set $N=M/x_1M$, $J=(x_1,\ldots,x_d)$ and $\overline{J}$ is image of $J$ is $A/(x_1)$. Then we have  
$$0 \rt \frac{\mathfrak{m}^{2}M:x_1}{\mathfrak{m}M}
\xrightarrow{f} \frac{\mathfrak{m}^{2}M}{J\mathfrak{m}M}
\xrightarrow{g} \frac{\mathfrak{m}^{2}{N}}{\overline{J}\mathfrak{m}{N}}\rt 0.$$
Here, $f(a+\mathfrak{m}M)=ax_1+J\mathfrak{m}M$ and $g$ is reduction modulo $x_1$.

\s\label{exact1} Let $(A,\mathfrak{m})$ be a \CM\ local ring of dimension one and $M$ a maximal \CM\ $A$-module. Let $x$ be a superficial element of $M$. Set $N=M/xM$. Then we have $$0\rt \frac{\mathfrak{m}^2M:x}{\mathfrak{m}^2M}\xrightarrow{f} \frac{\mathfrak{m}^2M}{x\mathfrak{m}^2M}\xrightarrow{g}\frac{\mathfrak{m}^2N}{0}\rt 0.$$
Here, $f(a+\mathfrak{m}^2M)=ax+x\mathfrak{m}^2M$ and $g$ is reduction modulo $x$.

The following  result is well known, but we will use this many times. For the convenience of the reader we state it

\s \label{overline{G(M)}} Let $(A,\mathfrak{m})=(Q/(f),\mathfrak{n}/(f))$, with $(Q,\mathfrak{n})$  a regular local ring of dimension $d+1$ and  $f\in \mathfrak{n}^i\setminus\mathfrak{n}^{i+1}$. Now if $M$ is a maximal \CM\ $A$-module with red$(M)\leq2$. Let $\underline{x}=x_1,\ldots,x_d$ be sufficiently general linear forms in $\mathfrak{n}/\mathfrak{n}^2$. Set $S=G_{\mathfrak{n}}(Q)$, $R=S/(\underline{x^*})S$ then $R\cong k[T]$ and $G(A)/(\underline{x})G(A)\cong R/(T^s)$ for some $s\geq2$.\\ Now consider $\overline{G(M)}=G(M)/(\underline{x^*})G(M)$. Then
$$\overline{G(M)}=M/\mathfrak{m}M \oplus \mathfrak{m}M/(\mathfrak{m}^2M+(\underline{x})M) \oplus \mathfrak{m}^2M/(\mathfrak{m}^3M+(\underline{x})\mathfrak{m}M)$$
Its Hilbert series is $\mu(M)+\alpha z+\beta z^2$ where $\beta\leq \alpha\leq \mu(M)$, because it is an $R$-module which is also $R/(T^s)$-module and it is generated in degree zero.

\s \label{RR-2} Let $(A,\mathfrak{m})$ be a \CM\ local ring of dimension $d$ and $M$ be a finite $A$-module with $\depth{M}\geq 2$. Let $x$ be an $M$-superficial element. Set $N=M/xM$. Then for $n\geq 0$ we have exact sequence  (see \cite[2.2]{Pu2})
\begin{equation}
0\rt \frac{(\mathfrak{m}^{n+1}M:x)}{\mathfrak{m}^nM}\rt \frac{\widetilde{\mathfrak{m}^nM}}{\mathfrak{m}^nM}\rt \frac{\widetilde{\mathfrak{m}^{n+1}M}}{\mathfrak{m}^{n+1}M} \rt \frac{\widetilde{\mathfrak{m}^{n+1}N}}{\mathfrak{m}^{n+1}N}.
\end{equation}

In particular, we have exact sequence
\begin{equation}
0\rt \frac{\widetilde{\mathfrak{m}M}}{\mathfrak{m}M}\rt \frac{\widetilde{\mathfrak{m}N}}{\mathfrak{m}N}.
\end{equation}

If $\depth{M}=1$, then for all $n\geq 0$ we have following exact sequence 
\begin{equation}
0\rt \frac{(\mathfrak{m}^{n+1}M:x)}{\mathfrak{m}^nM}\rt \frac{\widetilde{\mathfrak{m}^nM}}{\mathfrak{m}^nM}\rt \frac{\widetilde{\mathfrak{m}^{n+1}M}}{\mathfrak{m}^{n+1}M}.
\end{equation}

 The next result is a basic fact from linear algebra.
\begin{proposition}\label{vector} Let $V$ be a vector space of dimension $d\geq2 $ over an infinite field $k$. Let $V_1,\ldots V_n$ be finitely many proper subspaces of $V$. If dim$_k(V_i)\leq $ dim$_kV-2$, then there exists a subspace $H=ka\oplus kb$ where $a,b\in V$ such that $H\cap V_i=0$ for $i=0,\ldots, n$.
\end{proposition}

Following result is well-known (for instance see \cite[Lemma 2.33]{Mishra})
\begin{lemma}\label{le1}
	Let $(A,\mathfrak{m})$ be a complete hypersurface ring of dimension $d$ with infinite residue field and multiplicity $e(A)=e$. Let  $M$ be a MCM module. Let $\underline{x}=x_1,\ldots,x_d$ be a maximal $A$-superficial sequence. If $M$ has no free summand, then $M_d=M/(x_1,\ldots,x_d)M$ also has no free summand.
\end{lemma}

{\bf Convention}:
Let $M$ be a maximal Cohen-Macaulay  module of dimension $d$ and $\underline{x}=x_1,\ldots,x_d$ be a maximal $\phi$-superficial sequence, then

$M_0=M$ and $M_t=M/(x_1,\ldots,x_t)M$ for $t=1,\ldots,d$.

\section{\bf  Main Result}
We first consider the case when MCM module $M$ has no free summand.
\begin{theorem}
	Let $(A,\mathfrak{m})$ be a   hypersurface ring of dimension $d$ with $e(A)=3$. Let $M$ be an MCM module with no free summand. Now if  $\mu(M)=4$, then $\depth{G(M)}\geq d-3$.
\end{theorem}
\begin{proof}
	We may assume $A$ is complete with infinite residue field (see \ref{compl_and_inf}).
	Since $e(A)=3$,  we can take $(A,\mathfrak{m})=(Q/(g),\mathfrak{n}/(g))$ where $(Q,\mathfrak{n})$ is a regular local ring of dimension $d+1$ and $g\in \mathfrak{n}^3\setminus\mathfrak{n}^4$. This implies that $h_A(z)=1+z+z^2$ and $\mathfrak{m}^3=J\mathfrak{m}^2$, where $J$ is a minimal reduction of $A$. 
	
	Let dim$M\geq1$ and $0\rt Q^4\xrightarrow{\phi} Q^4\rt M\rt 0$ be a minimal presentation of $M$. Let $\underline{x}=x_1,\ldots,x_d$ be a maximal $\phi$-superficial sequence (see \ref{phi}). Set $M_d=M/\underline{x}M$ and $(Q',(y))=(Q/(\underline{x}),\mathfrak{n}/(\underline{x}))$. Note $A/(\underline{x})A=Q'/(y^3)$.
	
	Clearly, $Q'$ is DVR and so $M_d\cong Q'/(y^{a_1})\oplus Q'/(y^{a_2})\oplus Q'/(y^{a_3})\oplus Q'/(y^{a_4})$.
	
	From the Lemma \ref{le1}, $M_d$ has no free summand. So, we can assume that $1\leq a_1\leq a_2\leq a_3\leq a_4\leq 2$.

	This implies $4\leq e(M)\leq 8$. We consider all cases separately:
	
	{\bf Case(1): $e(M)=4$.}\\ 
	In this case $M_d\cong Q'/(y)\oplus Q'/(y)\oplus Q'/(y)\oplus Q'/(y)$. This implies $h_{M_d}=4$, so $e(M_d)=\mu(M_d)=4$.\\ For dim$M\geq1$, $e(M)=\mu(M)=4$. Also, notice that $i(M)=i(M_d)=1$. Since $e(M)=\mu(M)=4$, $M$ is an Ulrich module (see \ref{Ulrich}).  This implies that $G(M) $ is Cohen-Macaulay and $h_M(z)=4$ (see \cite[Theorem 2]{PuMCM}). 
	
	{\bf Case(2): $e(M)=5$.}\\
	In this case   $M_d\cong Q'/(y)\oplus Q'/(y)\oplus Q'/(y)\oplus Q'/(y^2)$ and   $h_{M_d}(z)=4+z$.

	This implies  $i(M)=i(M_d)=1, \mu(M)=4$ and $e(M)=e(M_d)=5$. 
	Now for dim$M\geq 1$ we have $\depth{G(M)}\geq d-1$ (see \cite[Theorem 1.1]{Mishra}). Also, we have two cases:\\ First case when  $h_M(z)=4+z$. In this case $G(M)$ is \CM.\\
	Second case when $h_M(z)=4+z^2$. In this case depth$G(M)=d-1$.

	{\bf Case(3): $e(M)=6$.}\\
	In this case  $M_d\cong Q'/(y)\oplus Q'/(y)\oplus Q'/(y^2)\oplus Q'/(y^2)$ and this implies  $h_{M_d}(z)=4+2z$.\\
	We first consider the  case when dim$M=4$ because if dim$M\leq3$ there is nothing to prove.\\
	Let $\underline{x}=x_1,x_2,x_3,x_4$ be a maximal $\phi$-superficial sequence. Set $M_1=M/x_1M$, $M_2=M/(x_1,x_2)M$, $M_3=M/(x_1,x_2,x_3)M$, $J_1=(x_2,x_3,x_4)$  and $J_2=(x_3,x_4)$.\\ 
	Since dim$M_3=1$ we can write $h$-polynomial of $M_3$ as $h_{M_3}(z)=4+(\rho_0(M_3)-\rho_1(M_3))z+\rho_1(M_3)z^2$
	where $\rho_n(M_3)=\ell(\mathfrak{m}^{n+1}M_3/{x_4\mathfrak{m}^nM_3})$ (see \ref{rho_formula}).
	So we have $\rho_0(M_3)=2$ and since all the coefficients of $h_M$ are non-negative (see \ref{d=1}(1)), so possible values of $\rho_1(M_3)$ are $0$, $1$ and $2$.\\
	{\bf Subcase(i): $\rho_1(M_3)=0$.}\\
	In this case, $M_3$ has minimal multiplicity and $h_{M_3}(z)=4+2z$. This implies  $G(M_3)$ is \CM. By Sally-descent $G(M)$ is \CM.\\
	{\bf Subcase(ii): $\rho_1(M_3)=1$. }\\
	In this case,  $h_{M_3}(z)=4+z+z^2$. This implies  depth$G(M_3)=0$, because $h_{M_3}(z)\ne h_{M_4}(z)$ (see \ref{Property}).\\
	Since dim$M_2=2$, we have (see \ref{Property})
	$$e_2(M_2)=e_2({M_3})-\sum b_i(x_3,M_2),$$
	where $b_i(x_3,M_2)=\ell(\mathfrak{m}^{i+1}M_2:x_3/\mathfrak{m}^iM_2)$.
	We know that $e_2(M_2) $ and $\sum b_i(x_3,M_2)  $ are non-negative integers. Note we  have  $e_2({M_3})=1$. This implies $\sum b_i(x_3,M_2) \leq 1 $.
	
	Since red$(M)\leq2$, from exact sequence 
	\begin{align*}
	0 \rt \frac{\mathfrak{m}^nM_2:J_2}{\mathfrak{m}^{n-1}M_2}\rt \frac{\mathfrak{m}^nM_2:x_3}{\mathfrak{m}^{n-1}M_2}
	 \rt \frac{\mathfrak{m}^{n+1}M_2:x_3}{\mathfrak{m}^{n}M_2}
	\rt \frac{\mathfrak{m}^{n+1}M_2}{J_2\mathfrak{m}^nM_2}
	  \rt \frac{\mathfrak{m}^{n+1}{M_3}}{x_4\mathfrak{m}^n{M_3}}\rt 0
	\end{align*}
	we get, if $b_1(x_3,M_2)=0$ then  $b_i(x_3,M_2)=0$ for all $i\geq2$. So, $\sum b_i(x_3,M_2)  \ne 0$ implies $b_1(x_3,M_2)=1$. \\
	Now we have two cases.\\
	{\bf Subcase (ii).(a):} When $b_1(x_3,M_2)=0$.\\ This implies  $\depth{G(M_2)}\geq 1$ (see \ref{Property}). In fact $\depth{G(M_2)}=1$ otherwise $G(M_2)$ is \CM. This  is not possible because $\depth{G(M_3)}=0$. Now by Sally-descent depth$G(M)=3$ and $h_M(z)=4+z+z^2$. \\
	{\bf Subcase (ii).(b):} When  $b_1(x_3,M_2)=1$.\\ So, $\depth{G(M_2)}=0$ (see \ref{Property}). In this case $h_{M_2}(z)=h_{M_3}(z)-(1-z)^2z=4+3z^2-z^3$ (see \ref{mod-sup}).
	
	From the above exact sequence we get
	\begin{align} \label{M_2}
	0\rt \mathfrak{m}^{2}M_2:x_3/\mathfrak{m}M_2
	\rt \mathfrak{m}^{2}M_2/J_2\mathfrak{m}M_2
	\rt \mathfrak{m}^{2}{M_3}/x_4\mathfrak{m}{M_3}\rt 0.
	\end{align}
	
	So, $\ell(\mathfrak{m}^{2}M_2/J_2\mathfrak{m}M_2)=\rho_1(M_3)+b_1(x_3,M_2)=2$.\\
	Since dim$M_1=3$,  from short exact sequence (see \ref{exact d}) $$0\rt \mathfrak{m}^2M_1:x_2/\mathfrak{m}M_1\rt {\mathfrak{m}^2M_1}/J_1\mathfrak{m}M_1\rt \mathfrak{m}^2M_2/{J_2}\mathfrak{m}M_2\rt 0$$
	we have $\ell(\mathfrak{m}^2M_1/J_1\mathfrak{m}M_1)=2+\ell(\mathfrak{m}^2M_1:x_2/\mathfrak{m}M_1)$.
	We also have
		$$G(M_1)/(x_2^*,x_3^*,x_4^*,)G(M_1)=M_1/\mathfrak{m}M_1 \oplus \mathfrak{m}M_1/J_1M_1\oplus \mathfrak{m}^2M_1/J_1\mathfrak{m}M_1.$$
	By considering its Hilbert series we get   $\ell(\mathfrak{m}^2M_1/J_1\mathfrak{m}M_1)\leq 2$, because in this case $\ell(\mathfrak{m}M_1/J_1M_1)=2$ (see \ref{overline{G(M)}}). Therefore we have $\ell(\mathfrak{m}^2M_1/J_1\mathfrak{m}M_1)= 2$.
	
	We also know that $\mathfrak{m}^2M_1\sub J_1M_1$.\\
	So in this case 
	$$\delta =\sum\ell(\mathfrak{m}^{n+1}M_1\cap J_1M_1/J_1\mathfrak{m}^nM_1)= 2$$
	
	We know that if $\delta\leq2$ then $\depth{G(M)}\geq d-\delta$ (see \cite[Theorem 5.1]{apprx}). So we have $\depth{G(M_1)}\geq 1$. Also notice that $\depth{G(M_1)}=1$, because $\depth{G(M_2)}=0$.\\ 
	By Sally-descent  $\depth{G(M)}=2$ and $h_M(z)=4+3z^2-z^3$. \\
	{\bf Subcase(iii): $\rho_1(M_3)=2$.}\\
	In this case $h_{M_3}(z)=4+2z^2$. This implies depth$G(M_3)=0$, because $h_{M_3}(z)\ne h_{M_4}(z) $ (see \ref{Property}).\\
	Since dim$M_2=2$, we have (see \ref{Property})
	$$e_2(M_2)=e_2({M_3})-\sum b_i(x_3,M_2).$$
	We know that $e_2(M_2) $ and $\sum b_i(x_3,M_2)  $ are non-negative integers. In this case we also have  $e_2({M_3})=2$. This implies $\sum b_i(x_3,M_2) \leq 2$.\\
	By an argument given in subcase (ii) we know that  $b_1(x_3,M_2)=0$ implies all $b_i(x_3,M_2)=0$.\\
	Now we have two cases.\\
	{\bf Subcase (iii).(a):} When $b_1(x_3,M_2)=0$.\\
	So, in this case $\depth{G(M_2)}=1$ (see \ref{Property}). Also notice $\depth{G(M_2)}\ne 2$ because $\depth{G(M_3)}=0$. By Sally-descent $\depth{G(M)}=3$ and $h_M(z)=4+2z^2$.\\  
	{\bf Subcase (iii).(b):} When $b_1(x_3,M_2)\neq0$.\\ Now from the exact sequence (\ref{M_2}) we get $\ell(\mathfrak{m}^2M_2/J_2\mathfrak{m}M_2)=\rho_1(M_3)+b_1(x_3,M_2)\geq3$.\\  Now consider  $\overline{G(M_2)}=G(M_2)/(x_3^*,x_4^*)G(M_2)$. Then we get 
	$$\overline{G(M_2)}=M_2/\mathfrak{m}M_2\oplus \mathfrak{m}M_2/J_2M_2\oplus\mathfrak{m}^2M_2/J_2\mathfrak{m}M_2$$ 
	Its  Hilbert series is $4+2 z+(\rho_1(M_3)+b_1(x_3,M_2))$,  because $\ell(\mathfrak{m}M_2/J_2M_2)=2$. But this is not  a possible Hilbert series (see \ref{overline{G(M)}}). Therefore  the case when $b_1(x_3,M_2)\neq0$ is not possible.
	
	Now assume dim$M\geq5$ and $\underline{x}=x_1,\ldots,x_d$ a maximal $\phi$-superficial sequence. Set $M_{d-4}=M/(x_1,\ldots,x_{d-4})M$. We now have three cases.\\
	First case when $G(M_{d-4})$ is \CM. By Sally-descent $G(M)$ is \CM\ and $h_M(z)=4+2z$.\\
	Second case when $\depth{G(M_{d-4})}=3$. By Sally-descent $\depth{G(M)}=d-1$ and $h_M(z)=4+z+z^2$ or $h_M(z)=4+2z^2$.\\
	Third case when  $\depth{G(M_{d-4})}=2$. By Sally-descent $\depth{G(M)}=d-2$ and $h_M(z)=4+3z^2-z^3$.

	{\bf Case(4): $e(M)=7$.}\\
	In this case  $M_d\cong Q'/(y)\oplus Q'/(y^2)\oplus Q'/(y^2)\oplus Q'/(y^2)$ and $h_{M_d}(z)=4+3z$.\\
	We first consider the case when dim$M=4$ because if dim$M\leq3$ there is nothing to prove.\\
	Let $\underline{x}=x_1,x_2,x_3,x_4$ be a maximal $\phi$-superficial sequence. Set $M_1=M/x_1M$, $M_2=M/(x_1,x_2)M$, $M_3=M/(x_1,x_2,x_3)M$, $J_1=(x_2,x_3,x_4)$, $J_2=(x_3,x_4)$ and $J=(x_1,x_2,x_3,x_4)$.\\ 
	We first prove two claims:
	
	{\bf Claim(1):} $\widetilde{\mathfrak{m}^iM_3}=\mathfrak{m}^iM_3$ for all $i\geq2.$\\
	{\bf Proof of Claim:} Since we have $\mathfrak{m}^{n+1}M_3=x_4\mathfrak{m}^nM_3$ for all $n\geq2$. So $(\mathfrak{m}^{n+1}M_3:x_4)=\mathfrak{m}^nM_3$ for all $n\geq2$. We have  exact sequence (see \ref{RR-2}) $$0\rt \mathfrak{m}^{n+1}M_3:x_4/\mathfrak{m}^nM_3\rt \widetilde{\mathfrak{m}^nM_3}/\mathfrak{m}^nM_3\rt \widetilde{\mathfrak{m}^{n+1}M_3}/\mathfrak{m}^{n+1}M_3.$$ We also know that for $n\ggg 0$, $\widetilde{\mathfrak{m}^nM_3}=\mathfrak{m}^nM_3$. By using these facts it is clear that $\widetilde{\mathfrak{m}^iM_3}=\mathfrak{m}^iM_3$ for all $i\geq2.$
	
	{\bf Claim(2):} $\ell(\widetilde{\mathfrak{m}M_3}/\mathfrak{m}M_3)\leq1$.\\
	{\bf Proof of the claim:} Since $\mu(M_3)=4$,  we have $\ell(\widetilde{\mathfrak{m}M_3}/\mathfrak{m}M_3)\leq4$.\\ If $\ell(\widetilde{\mathfrak{m}M_3}/\mathfrak{m}M_3)=4$ then $\widetilde{\mathfrak{m}M_3}=M_3$. So $\mathfrak{m}^3M_3=\mathfrak{m}^2M_3$ because we know that $\widetilde{\mathfrak{m}^iM_3}=\mathfrak{m}^iM_3$ for all $i\geq2$. So, from here we have $\mathfrak{m}^2M_3=0$ which is a contradiction. Therefore $\ell(\widetilde{\mathfrak{m}M_3}/\mathfrak{m}M_3)\leq3$.\\
	If possible assume that $\ell(\widetilde{\mathfrak{m}M_3}/\mathfrak{m}M_3)>1$, so we have $M_3=\langle m_1,m_2,l_1,l_2 \rangle$ where $l_1,l_2\in \widetilde{\mathfrak{m}M_3}\setminus\mathfrak{m}M_3$. This implies $l_i\mathfrak{m}\sub \widetilde{\m^2M_3}=\mathfrak{m}^2M_3$ for $i=1,2$. Now if we set $\mathfrak{m}'=\mathfrak{m}/(x_1,x_2,x_3,x_4)$ then $\mathfrak{m}'$ is a principal ideal. We also know that $\ell(\mathfrak{m}M_4)=\ell(\mathfrak{m}'M_4)$ and $\mathfrak{m}^2M_4=\mathfrak{m'}^2M_4=0$. From here  we get  $\ell({\mathfrak{m}M_4})\leq2$. This is a contradiction because we know that $\ell({\mathfrak{m}M_4})=3$. So, $\ell(\widetilde{\mathfrak{m}M_3}/\mathfrak{m}M_3)\leq1$.\\
	Now we have two cases.\\
	{\bf Subcase (i):} When $\widetilde{\mathfrak{m}M_3}=\mathfrak{m}M_3$.\\ So, we have $\widetilde{\mathfrak{m}^iM_3}=\mathfrak{m}^iM_3$ for all $i$, because we know that $\widetilde{\mathfrak{m}^iM_3}=\mathfrak{m}^iM_3$ for all $i\geq2$ (from claim(1)). So in this case $\depth{G(M_3)}=1$, i.e. $G(M_3)$ is \CM. By Sally-descent $G(M)$ is \CM\ and $h_M(z)=4+3z$.\\
	{\bf Subcase (ii):}  When $\ell(\widetilde{\mathfrak{m}M_3}/\mathfrak{m}M_3)=1$. In this subcase from \ref{RR-2}(3) and the fact $\widetilde{\mathfrak{m}^2M_3}=\mathfrak{m}^2M_3$ we get $\ell(\mathfrak{m}^2M_3:x_4/\mathfrak{m}M_3)=1$.  \\
	Since dim$M_3=1$ we can write $h$-polynomial of $M_3$ as $h_{M_3}(z)=4+(\rho_0(M_3)-\rho_1(M_3))z+\rho_1(M_3)z^2$ where $\rho_n(M_3)=\ell(\mathfrak{m}^{n+1}M_3/{x_4\mathfrak{m}^nM_3})$ (see \ref{rho_formula}). We have  $\rho_0(M_3)=\ell(\mathfrak{m}M_2/x_3M_2)=3$ and coefficients of $h_{M_3}$ are non-negative (see \ref{d=1}(1)). 
	
	From short exact sequence (see \ref{exact d})
	$$0\rt \mathfrak{m}^2M_3:x_4/\mathfrak{m}M_3\rt \mathfrak{m}^2M_3/x_4\mathfrak{m}M_3\rt \mathfrak{m}^2M_4/0\rt 0$$
	we have $\rho_1(M_3)=b_1(x_4,M_3)$ because $\mathfrak{m}^2M_4=0$. 
	
So we have $\rho_1(M_3)=b_1(x_4,M_3)=\ell(\mathfrak{m}^2M_3:x_4/\mathfrak{m}M_3)=1$. This implies $\depth{G(M_3)}=0$, because $h_{M_3}(z)\ne h_{M_4}(z)$ (see \ref{Property}). We have  $h_{M_3}(z)=4+2z+z^2.$\\
	We also have $\widetilde{\mathfrak{m}^2M_3}=x_4\widetilde{\mathfrak{m}M_3}$. In fact, if $a\in\widetilde{\mathfrak{m}^2M_3}  $ then we can write $a=xp$ because $\widetilde{\mathfrak{m}^2M_3}=\mathfrak{m}^2M_3\sub (x_4)M_3$. This implies that $p\in \widetilde{\mathfrak{m}M_3}$ because $(\widetilde{\mathfrak{m}^2M_3}:x_4)=\widetilde{\mathfrak{m}M_3}$. So we have $\widetilde{\mathfrak{m}^{i+1}M_3}=x_4\widetilde{\mathfrak{m}^iM_3}$ for all $i\geq 1$ because $\mathfrak{m}^{i+1}M_3=x_4\mathfrak{m}^iM_3$ for $i\geq2$. This implies that $\widetilde{G(M_3)}$ has minimal multiplicity and $\widetilde{h_{M_3}}(z)=3+4z$.\\
	Here we have two cases.\\
	{\bf Subcase (ii).(a):} When $\depth{G(M_2)}\ne 0$.\\ Then we have $\depth{G(M_2)}=1$ because $\depth{G(M_3)}=0$. By Sally-descent $\depth{G(M)}=3$ and $h_M(z)=4+2z+z^2$.\\
	{\bf Subcase (ii).(b): } When $\depth{G(M_2)}=0$.\\ 
	As $e_2(M_3)=1$, we get $\sum b_i(x_3,M_2)\leq 1$. So we have $b_1(x_3,M_2)=\ell(\mathfrak{m}^2M_2:x_3/\mathfrak{m}M_2)\neq 0$ (see \ref{Property}). Because $b_1(x_3,M_2)=0$ implies all $b_i(x_3,M_2)=0$, so $\depth G(M)\ne 0$, a contradiction (same argument as in Subcase(ii) of Case(3)). \\
	From exact sequences (see \ref{RR-2})
	$$0\rt \frac{\mathfrak{m}^{n+1}M_2:x_3}{\mathfrak{m}^nM_2}\rt \frac{\widetilde{\mathfrak{m}^nM_2}}{\mathfrak{m}^nM_2}\rt \frac{\widetilde{\mathfrak{m}^{n+1}M_2}}{\mathfrak{m}^{n+1}M_2}\rt \frac{\widetilde{\mathfrak{m}^{n+1}M_3}}{\mathfrak{m}^{n+1}M_3} $$
	and $$0 \rt \widetilde{\mathfrak{m}M_2}/\mathfrak{m}M_2\rt \widetilde{\mathfrak{m}M_3}/\mathfrak{m}M_3 $$
	we have  $$1\leq b_1(x_3,M_2)\leq \ell(\widetilde{\mathfrak{m}M_2}/\mathfrak{m}M_2)\leq \ell(\widetilde{\mathfrak{m}M_3}/\mathfrak{m}M_3)=1. $$
	This implies that $\ell(\widetilde{\mathfrak{m}M_2}/\mathfrak{m}M_2)=1$ and $b_1(x_3,M_2)=1$. From here we also get 	 $\widetilde{\mathfrak{m}^nM_2}=\mathfrak{m}^nM_2$ for all $n\geq2$,
	because from claim(1) we know that $\widetilde{\mathfrak{m}^nM_3}=\mathfrak{m}^nM_3$ for all $n\geq2$ (see \ref{RR-2}(1)).
	\\
	From the exact sequence (see \ref{exact d})
	$$ 0\rt \mathfrak{m}^{2}M_2:x_3/\mathfrak{m}M_2
	\rt \mathfrak{m}^{2}M_2/J_2\mathfrak{m}M_2
	\rt \mathfrak{m}^{2}{M_3}/x_4\mathfrak{m}{M_3}\rt 0$$
	we get $\ell(\mathfrak{m}^{2}M_2/J_2\mathfrak{m}M_2)=\rho_1+b_1(x_3,M_2)=2$.\\
	In this case we also have	
	$h_{M_2}(z)=h_{M_3}(z)-(1-z)^2z=4+z+3z^2-z^3$ (see \ref{mod-sup}).\\
	Since we have  $\ell(\widetilde{\mathfrak{m}M_2}/\mathfrak{m}M_2)=\ell(\widetilde{\mathfrak{m}M_3}/\mathfrak{m}M_3) $ and $\widetilde{\mathfrak{m}^iM_3}=\mathfrak{m}^iM_3$ for all $i\geq2$, from \cite[2.1]{Pu2}
	we get $$\overline{\widetilde{\mathfrak{m}^iM_2}}=\widetilde{\mathfrak{m}^iM_3}\ \ \text{ for all }  i\geq 1.$$	
	
	So, $\widetilde{G(M_2)}/x_3^*\widetilde{G(M_2)}=\widetilde{G(M_3)}$, this implies that $\widetilde{G(M_2)}$ is \CM \ and $\widetilde{h_{M_2}}(z)=3+4z$.\\
	From exact sequence  (see \ref{exact d})
	\begin{align}\label{M_1}
	0\rt \mathfrak{m}^{2}M_1:x_2/\mathfrak{m}M_1
	\rt \mathfrak{m}^{2}M_1/J_1\mathfrak{m}M_1
	\rt \mathfrak{m}^{2}{M_2}/J_2\mathfrak{m}{M_2}\rt 0
	\end{align}

	we have , if $\ell(\mathfrak{m}^2M_1:x_2/\mathfrak{m}M_1)=0$ then
	
	$\ell(\mathfrak{m}^2M_1/J_1\mathfrak{m}M_1)=2$, because $\ell(\mathfrak{m}^2M_2/J_2\mathfrak{m}M_2)=2$.\\
	{\bf Subcase (ii).(b).(1):} $\ell(\mathfrak{m}^2M_1:x_2/\mathfrak{m}M_1)=0$.\\  Consider $$\delta=\sum\ell(\mathfrak{m}^{n+1}M_1\cap J_1M_1/J_1\mathfrak{m}^nM_1).$$ 
	We know that if $\delta\leq2$ then $\depth{G(M_1)}\geq d-\delta$ (see \cite[Theorem 5.1]{apprx}).
	Since  $\delta=2$,  $\depth{G(M_1)}\geq1$. Notice that here $\depth{G(M_1)}=1$, because $\depth{G(M_2)}=0$. By Sally-descent $\depth{G(M)}=2$ and $h_M(z)= 4+z+3z^2-z^3$.\\
	{\bf Subcase (ii).(b).(2):} $\ell(\mathfrak{m}^2M_1:x_2/\mathfrak{m}M_1)\neq0$.\\ This implies $\depth{G(M_1)}=0$ (see \ref{Property}). 
	
	Now from exact sequences (see \ref{RR-2})
	$$0\rt \frac{\mathfrak{m}^{n+1}M_1:x_2}{\mathfrak{m}^nM_1}\rt \frac{\widetilde{\mathfrak{m}^nM_1}}{\mathfrak{m}^nM_1}\rt \frac{\widetilde{\mathfrak{m}^{n+1}M_1}}{\mathfrak{m}^{n+1}M_1}\rt \frac{\widetilde{\mathfrak{m}^{n+1}M_2}}{\mathfrak{m}^{n+1}M_2} $$
	and $$0 \rt \widetilde{\mathfrak{m}M_1}/\mathfrak{m}M_1\rt \widetilde{\mathfrak{m}M_2}/\mathfrak{m}M_2 $$
	
	we get $$1\leq \ell(\mathfrak{m}^2M_1:x_2/\mathfrak{m}M_1)\leq \ell(\widetilde{\mathfrak{m}M_1}/\mathfrak{m}M_1)\leq \ell(\widetilde{\mathfrak{m}M_2}/\mathfrak{m}M_2)=1. $$
	This implies $\ell(\mathfrak{m}^2M_1:x_2/\mathfrak{m}M_1)=\ell(\widetilde{\mathfrak{m}M_1}/\mathfrak{m}M_1)=1$. From here we also get  $\widetilde{\mathfrak{m}^iM_1}=\mathfrak{m}^iM_1$ for all $i\geq2$, because $\widetilde{\mathfrak{m}^iM_2}=\mathfrak{m}^iM_2$ for all $i\geq2$ (see \ref{RR-2}(1)). \\  From the short exact sequence (\ref{M_1})
	we have $\ell(\mathfrak{m}^2M_1/J_1\mathfrak{m}M_1)=3$.\\ Now since $\ell(\widetilde{\mathfrak{m}M_1}/\mathfrak{m}M_1)=\ell(\widetilde{\mathfrak{m}M_2}/\mathfrak{m}M_2)$ and $\widetilde{\mathfrak{m}^iM_1}=\mathfrak{m}^iM_1$ for all $i\geq2$, from \cite[2.1]{Pu2} we get $$\overline{\widetilde{\mathfrak{m}^iM_1}}=\widetilde{\mathfrak{m}^iM_2}\ \ \text{for all } \ i\geq 1.$$
	
	So $\widetilde{G(M_1)}/x_2^*\widetilde{G(M_1)}=\widetilde{G(M_2)}$, this implies that $\widetilde{G(M_1)}$ is \CM \ and $\widetilde{h_{M_1}}(z)=3+4z$.\\
	We can write the $h$-polynomial of $M_1$ as  $h_{M_1}(z)=\widetilde{h_{M_1}}(z)+(1-z)^4$.\\
	Consider  $$G(M)/(x_1^*x_2^*,x_3^*,x_4^*)G(M)=M/\mathfrak{m}M\oplus \mathfrak{m}M/JM\oplus \mathfrak{m}^2M/J\mathfrak{m}M.$$ After looking at its Hilbert series we get $\ell(\mathfrak{m}^2M/J\mathfrak{m}M)\leq3$, because $\ell(\mathfrak{m}M/JM)=3$ (see \ref{overline{G(M)}}).
	
	We have short exact sequence (see \ref{exact d})
	$$0\rt \mathfrak{m}^2M:x_1/\mathfrak{m}M\rt \mathfrak{m}^2M/J\mathfrak{m}M\rt \mathfrak{m}^2M_1/{J_1}\mathfrak{m}M_1\rt 0.$$
	This implies that $\ell(\mathfrak{m}^2M/J\mathfrak{m}M)\geq3$.\\ So we have $$(\mathfrak{m}^2M:x_1)=\mathfrak{m}M \ \text{and}\ \ell(\mathfrak{m}^2M/J\mathfrak{m}M)=3. $$
	Now we first prove a claim.\\
	{\bf Claim:} $\widetilde{\mathfrak{m}M}=\mathfrak{m}M$.\\
	{\bf Proof of the claim:} 	If $\widetilde{\mathfrak{m}M}\neq \mathfrak{m}M$.
	From exact sequence (see \ref{RR-2})
	$$0 \rt \widetilde{\mathfrak{m}M}/\mathfrak{m}M\rt \widetilde{\mathfrak{m}M_1}/\mathfrak{m}M_1 $$
	we get $\ell(\widetilde{\mathfrak{m}M}/\mathfrak{m}M)=1$, because $\ell(\widetilde{\mathfrak{m}M}/\mathfrak{m}M)\leq \ell(\widetilde{\mathfrak{m}M_1}/\mathfrak{m}M_1)=1$. Since  in this case    
	$\ell(\widetilde{\mathfrak{m}M}/\mathfrak{m}M)= \ell(\widetilde{\mathfrak{m}M_1}/\mathfrak{m}M_1)$ and $\widetilde{\mathfrak{m}^iM_1}={\mathfrak{m}^iM_1}$ for all $i\geq 2$, from \cite[2.1]{Pu2} we get
	$$\overline{\widetilde{\mathfrak{m}^iM}}=\widetilde{\mathfrak{m}^iM_1}\ \ \text{for all} \ i\geq 1.$$ This implies $\widetilde{G(M)}/x_1^*\widetilde{G(M)}=\widetilde{G(M_1)}$. So we get  $\widetilde{G(M)}$ is \CM \ and therefore $G(M)$ is generalised \CM.\\
	Let Ass$_{G(A)}G(M)=\{\mathcal{M},\mathcal{P}_1,\ldots,\mathcal{P}_s\}$, where $\mathcal{M}$ is maximal homogeneous ideal of $G(A)$ and $\mathcal{P}_i$'s are minimal primes in $G(A)$ (see \ref{ASSG}). Set $V=\mathfrak{m}/\mathfrak{m}^2$. We know that $\mathcal{P}_i\cap V\neq V$. Now if dim$\mathcal{P}_i\cap V=$ dim$V-1$, then dim$G(A)/\mathcal{P}_i\leq1$ and this is a contradiction as $\mathcal{P}_i$'s are minimal primes in $G(A).$\\
	Thus, dim$\mathcal{P}_i\cap V\leq$ dim$V-2$. So there exists $u^*,v^*\in V$ such that $H=ku^*+kv^*$ and $H\cap \mathcal{P}_i=0$ for $i=1,\ldots,s$ (see \ref{vector}). Thus if $\xi\in \mathfrak{m}$ such that $\xi^*\in H$ is non-zero then $\xi $ is a superficial element of $M$ (see \cite[Theorem 1.2.3]{Rossi}). Now since $\ell(\widetilde{\mathfrak{m}M}/\mathfrak{m}M)=1$, $$\widetilde{\mathfrak{m}M}=\mathfrak{m}M+Aa\ \text{for some }\ a\not\in \mathfrak{m}M$$
	If $x_1a\in \mathfrak{m}^2M$ then $a\in(\mathfrak{m}^2M:x_1)=\mathfrak{m}M$ and this is a contradiction.\\
	So $\overline{x_1a}\neq0$ in $\widetilde{\mathfrak{m}^2M}/\mathfrak{m}^2M$.\\ Now from the exact sequence (see \ref{RR-2})
	\begin{align}
	0\rt \mathfrak{m}^{2}M:x_1/\mathfrak{m}M\rt \widetilde{\mathfrak{m}M}/\mathfrak{m}M\rt \widetilde{\mathfrak{m}^{2}M}/\mathfrak{m}^{2}M\rt \widetilde{\mathfrak{m}^{2}M_1}/\mathfrak{m}^{2}M_1. 
	\end{align}
	we get $\ell(\widetilde{\mathfrak{m}^2M}/\mathfrak{m}^2M)=\ell(\widetilde{\mathfrak{m}M}/\mathfrak{m}M)=1$. So we have $\overline{ua}=\beta\overline{va}$ where $\beta$ is unit and $u,v$ are $M$-superficial elements. Now we have $(u-\theta v)a\in \mathfrak{m}^2M$ where $\theta\in A$ and $\overline{\theta}=\beta$ is a unit. Since $(u-\theta v)^*=u^*-\beta v^*$ is nonzero element in $H$, so $u-\theta v$ is $M$-superficial. This implies that $$a\in (\mathfrak{m}^2M:(u-\theta v))=\mathfrak{m}M$$
	This is a contradiction. So $\widetilde{\mathfrak{m}M}=\mathfrak{m}M$.\\
	Since $\widetilde{\mathfrak{m}M}=\mathfrak{m}M$, now  from exact sequence (see \ref{RR-2})
	\begin{align}
	0\rt \frac{\mathfrak{m}^{n+1}M:x_1}{\mathfrak{m}^nM}\rt \frac{\widetilde{\mathfrak{m}^nM}}{\mathfrak{m}^nM}\rt \frac{\widetilde{\mathfrak{m}^{n+1}M}}{\mathfrak{m}^{n+1}M}\rt \frac{\widetilde{\mathfrak{m}^{n+1}M_1}}{\mathfrak{m}^{n+1}M_1}, 
	\end{align}
	we get $\widetilde{\mathfrak{m}^iM}=\mathfrak{m}^iM$ for all $i$, because $\widetilde{\mathfrak{m}^iM_1}=\mathfrak{m}^iM_1$ for all $i\geq 2$. This implies that depth$G(M)\geq1.$ Notice that here $\depth{G(M)}=1$ because $\depth{G(M_1)}=0$. In this case $h_M(z) = 3+4z+(1-z)^4$.\\
	Now assume  dim$M\geq5$ and $\underline{x}=x_1,\ldots,x_d$  a maximal $\phi$-superficial sequence. set $M_{d-4}=M/(x_1,\ldots,x_{d-2})M$.  Then we have the following cases.\\
	First case when $G(M_{d-4})$ is \CM. By Sally-descent $G(M)$ is \CM\ and $h_M(z)=4+3z$.\\
	Second case when $\depth{G(M_{d-4})}=3$. By Sally-descent depth$G(M)=d-1$ and
	$h_M(z)=4+2z+z^2$.\\
	Third case when $\depth{G(M_{d-4})}=2$. By Sally-descent $\depth{G(M)}=d-2$ and $h_M(z)=4+z+3z^2-z^3$.\\
	Fourth case when $\depth{G(M_{d-4})}=1$. By Sally-descent $\depth{G(M)}=d-3$ and $h_M(z)=3+4z+(1-z)^4$.

		{\bf Case(5): $e(M)=8$.}\\
	In this case  $M_d\cong Q'/(y^2)\oplus Q'/(y^2)\oplus Q'/(y^2)\oplus Q'/(y^2)$ and $h_{M_d}(z)=4+4z$. So $e(M_d)=\mu(M_d)i(M_d)$. For dim$M\geq 1$,  this equality is preserved modulo any $\phi$-superficial sequence. This implies that $G(M)$ is \CM \ (see \cite[Theorem 2]{PuMCM}).
\end{proof}

From the above theorem we can conclude:
\begin{enumerate}
	\item If $e(M)=4 $ then $a_1=a_2=a_3=a_4=1$. In this case $M$ is an Ulrich module so $G(M)$ is \CM\ and $h_M(z)=4.$
	\item If $e(M)=5$ then $a_1=a_2=a_3=1,a_4=2$. In this case we have two cases:
	\begin{enumerate}
		\item  $G(M)$ is \CM\ if and only if  $h_M(z)=4+z$.
		\item  $\depth{G(M)}=d-1$ if and only if $h_M(z)=4+z^2$. 
	\end{enumerate}
	\item If $e(M)=6$ then $a_1=a_2=1,a_3=a_4=2$. In this case we have three cases:
	\begin{enumerate}
		\item  $G(M)$ is \CM\ if and only if $h_M(z)=4+2z$.
		\item  $\depth{G(M)}=d-1$ if and only if $h_M(z)=4+z+z^2$ or $h_M(z)=4+2z^2$.
		\item  $\depth{G(M)}=d-2$ if and only if $h_M(z)=4+3z^2-z^3$.
	\end{enumerate}
	\item If $e(M)=7$ then $a_1=1,a_2=a_3=a_4=2$. In this case we have four cases:
	\begin{enumerate}
		\item $G(M)$ is \CM\ if and only if $h_M(z)=4+3z$.
		\item $\depth{G(M)}=d-1$ if and only if $h_M(z)=4+2z+z^2$.
		\item $\depth{G(M)}=d-2$ if and only if $h_M(z)=4+z+3z^2-z^4$.
		\item $\depth{G(M)}=d-3$ if and only if $h_M(z)=3+4z+(1-z)^4$.
	\end{enumerate}
	\item If $e(M)=8$ then $a_1=a_2=a_3=a_4=2$. In this case $G(M)$ is \CM\ and $h_M(z)=4+4z.$
\end{enumerate}

\begin{corollary}
	Let $(A,\mathfrak{m})$ be a   hypersurface ring of dimension $d$ with $e(A)=3$. Let $M$ be an MCM module. Now if  $\mu(M)=4$, then $\depth{G(M)}\geq d-3$.
	\end{corollary}
\begin{proof}
	We may assume $A$ is complete with infinite residue field (see \ref{compl_and_inf}).
	Since $e(A)=3$,  we can take $(A,\mathfrak{m})=(Q/(g),\mathfrak{n}/(g))$ where $(Q,\mathfrak{n})$ is a regular local ring of dimension $d+1$ and $g\in \mathfrak{n}^3\setminus\mathfrak{n}^4$.\\
	Now we have two cases here.\\
	First case when $M$ has no free summand. In this case, from the above theorem $\depth{G(M)}\geq d-3$.\\
	Next case when $M$ has  free summand. In this case we can write $M\cong N\oplus A^s$ for some $s\geq 1$ and $N$ has no free summand. We assume $N\ne 0$, otherwise  $M$ is free and  $G(M)$ is \CM.\\
	Clearly, $N$ is a MCM $A$-module (see \cite[Proposition 1.2.9]{BH}). Notice that red$(N)\leq 2$, because red$(M)\leq 2$. Also $\mu(M)> \mu(N)$, so $\mu(N)\leq 3$.\\ 
	If $\mu(N)=1$ then we have a minimal presentation of $N$ as $0\rt Q\xrightarrow{a} Q\rt N\rt  0$, where $a\in \mathfrak{n}$. 
	This implies $N\cong Q/(a)Q$. So, $G(N)$ is \CM.\\
	If $\mu(N)=2$ then $\depth{G(N)}\geq d-1$ (from \cite[Theorem 1.4]{Mishra}).\\	
	If $\mu(N)=3$ then  $\depth{G(N)}\geq d-2$ (from \cite[Theorem 1.5]{Mishra}).\\ 
	We know that (see \cite[Proposition 1.2.9]{BH})
	
	depth$G(M)\geq $ min\{$\depth{G(N)}$, $\depth{G(A)}$\}$=$ $\depth{G(N)}$.\\
	So in this case $\depth{G(M)}\geq d-2$. 
	
\end{proof}

\section{Examples} 
We give examples of MCM modules $M$ with no free summand and satisfying $\mu(M)=4$, $\dim M=3$ and $\depth G(M)=0,1,2,3$.

Take $Q=k[[x,y,z,t]]$, $\mathfrak{n}=(x,y,z,t)$ 
\begin{enumerate}
	\item   $\phi= \begin{pmatrix}
	x & y & z & t\\
	x^2 & x^2 & 0 &0\\
	0 & 0 & x^2 & 0\\
	0 & 0 & 0  & x^2
	\end{pmatrix} $
	then
	$\depth{G(M)}=0$ because $\overline{e_1}$ where $e_1=(1,0,0,0)^T$ is an element of $\widetilde{\mathfrak{n}M}\setminus \mathfrak{n}M$. We have $x\overline{e_1}=-x^2\overline{e_2}$, $y\overline{e_1}=-x^2\overline{e_2}$, $z\overline{e_1}=-x^2\overline{e_3}$ and $t\overline{e_1}=-x^2\overline{e_4}$. These relations imply that  $x^2(x-y)\overline{e_i}=0$ for $i=1,2,3,4$. So,  $M$ is $Q/(x^2(x-y))$-module.  
	
	\item  $\phi = \begin{pmatrix}
	x & y & z & 0\\
	x^2 & x^2 & 0 &0\\
	0 & 0 & x^2 & 0\\
	0 & 0 & 0  & x^2
	\end{pmatrix} $ 
	then $\depth{G(M)}=1$.\\ Since $t^*$ is $G(M)-$regular and  after going modulo $t^*$, we get $\depth{G({N})}=0$, here $N=M/tM$.  Notice that $\overline{e_1}\in \widetilde{\mathfrak{n}N}\setminus\mathfrak{n}N$. Since $x^2(x-y)\overline{e_i}=0$ for $i=1,2,3,4$, this implies $M$ is $Q/(x^2(x-y))$-module.
	\item  $\phi= \begin{pmatrix}
	x   & y   & 0   & 0\\
	x^2 & x^2 & 0   &0\\
	0   & 0   & x^2 & 0\\
	0   & 0   & 0   & x^2
	\end{pmatrix} $ then it is clear that $z^*,t^*$ is maximal $G(M)$-regular sequence. So, $\depth{G(M)}=2$. In fact, if we set $N=M/(z,t)M$ then $\overline{e_1}\in \widetilde{\mathfrak{n}N}\setminus\mathfrak{n}N$. Also notice that $M$ is $Q/(x^2(x-y))$-module, because $x^2(x-y)\overline{e_i}=0$ for $i=1,2,3,4.$
	\item  $\phi= \begin{pmatrix}
	x & 0 & 0 & 0\\
	0 & x^2 & 0 &0\\
	0 & 0 & x^2 & 0\\
	0 & 0 & 0  & x^2
	\end{pmatrix} $ then  $G(M)$ is \CM.
	
\end{enumerate}

\end{document}